\documentclass[11pt, a4paper, reqno]{amsart}
\usepackage[a4paper, margin=3.75cm]{geometry}
\usepackage[T2A]{fontenc}
\usepackage[cp866]{inputenc}
\usepackage[russian,english]{babel}
\usepackage{amsfonts,amssymb,mathrsfs,amscd,comment}
\usepackage{amsthm,graphicx} % These include AMSLatex style commands
\def\NoBlackBoxes{\overfullrule0pt}

\NoBlackBoxes
\theoremstyle{plain}

\newtheorem{theoremL}{Theorem}

\theoremstyle{definition}

\usepackage[breaklinks]{hyperref}
\hypersetup{
unicode=true,
colorlinks=true,
linkcolor=blue,
citecolor=blue,
urlcolor=blue,
filecolor=blue,
bookmarksnumbered=true,
pdfstartview=FitH,
pdfhighlight=/N
}
\theoremstyle{main}
\let\savedef=\endproof
\def\endproof{~$\square$\savedef}
\def\bad{\spaceskip=0.33emplus0.6emminus0.15em\immediate\write5{\string\bad}}

\theoremstyle{plain}
\newtheorem{proposition}{Proposition}
\let\myh\widehat

\let\eps\varepsilon

\def\CC{\mathbb C}

\def\NN{\mathbb N}

\def\mycap{\operatorname{cap}}
\def\mcap{\operatorname{cap}}

\def\supp{\operatorname{supp}}

\def\<{\left\langle}
\def\>{\right\rangle}
\def\({\left(}
\def\){\right)}
\def\[{\left[}
\def\]{\right]}
\NoBlackBoxes

\def\bad{\spaceskip=0.33emplus0.6emminus0.15em\immediate\write5{\string\bad}}
\def\NN{\mathbb N}

\def\CC{\mathbb C}

\let\myh\widehat
\let\myt\widetilde

\def\({\left(}
\def\){\right)}
\def\[{\left[}
\def\]{\right]}
\def\<{\left\langle}
\def\>{\right\rangle}

\let\geq\geqslant

\let\leq\leqslant
\begin{document}

\selectlanguage{english}

\title{Two examples based on the properties of discrete measures}
\author[Sergey~P.~Suetin]{Sergey~P.~Suetin}
\address{Steklov Mathematical Institute of the Russian Academy of Sciences, Russia}
\email{suetin@mi-ras.ru}

\markright{Two examples}

\date{May 15, 2021}

\begin{abstract}
In the paper we represent two  examples which are based on the properties of discrete measures.

In the first part of the paper we prove that for each probability measure   $\mu$, $\supp{\mu}=[-1,1]$, which logarithmic potential  is  a continuous function on $[-1,1]$ there exists a (discrete) measure $\sigma=\sigma(\mu)$, $\supp{\sigma}=[-1,1]$, with the following property. Let $\{P_n(x;\sigma)\}$ be the sequence of polynomials orthogonal with respect to $\sigma$. Then $\dfrac1n\chi(P_n(\cdot;\sigma))\overset{*}\to\mu$, $n\to\infty$, where $\chi(\cdot)$ is zero counting measure for the corresponding polynomial.
The construction of the measure $\sigma$ is based of the properties of weighted Leja points.

In the second part we give an  example of a compact set and a sequence of discrete measures supported on that compact set with the following property. The sequence of measures converges in weak-$*$ topology to the equilibrium measure for the compact set but the corresponding sequence of the logarithmic potentials does not converge to the equilibrium potential in any neighbourhood of the compact set.

Bibliography:~\cite{SaTo97}~titles.
\end{abstract}

\maketitle

\section{}\label{s1}

\subsection{}\label{s1s1}
Let  $\Delta=[-1,1]$ and $\mu$ be a probability measure such that $\supp{\mu}=\Delta$ and the corresponding logarithmic potential $V^\mu(z)$ of $\mu$ is a continuous function on~$\Delta$.

For a measure  $\sigma$, $\supp{\sigma}=\Delta$, we denote by $P_n(x;\sigma)=x^n+\dotsb$ the monic polynomials which are orthogonal with respect to $\sigma$,
\begin{equation}
\int_{\Delta}P_n(x;\sigma)x^k\,d\sigma(x)=0,\quad k=0,\dots,n-1.
\label{2}
\end{equation}
For an arbitrary polynomial $Q\in\CC[z]$ we denote by $\chi(Q)$ the zero  counting measure of $Q$ (in accordance with their multiplicities of zeros):
$$
\chi(Q):=\sum_{\zeta:Q(\zeta)=0}\delta_\zeta.
$$
It is well-known~\cite{Rak77},~\cite{Rak82} that if $\sigma'>0$ a.e. on $\Delta$ then
\begin{equation}
\lim_{n\to\infty}\frac{P_{n+1}(z;\sigma)}{P_n(z;\sigma)}=
\frac12(z+(z^2-1)^{1/2}),\quad z\notin\Delta
\notag
\end{equation}
(here $(z^2-1)^{1/2}/z\to1$ as $z\to\infty$). Thus we have as $n\to\infty$
\begin{equation}
\frac1n\chi(P_n(\cdot;\sigma))\overset{*}\to \tau^{}_\Delta,\quad d\tau^{}_\Delta
=\frac1\pi\frac{dx}{\sqrt{1-x^2}}.
\label{4}
\end{equation}
Notice that under a weaker condition on $\sigma$ the relation~\eqref{4} was proven in~\cite{Kor61}.

The following statement holds.
\begin{proposition}\label{pro1}
Let  $\mu$ be a probability measure such that $\supp{\mu}=\Delta$, and the logarithmic potential $V^\mu(z)$ of $\mu$ is a continuous function on $\Delta$.   Then there exists a measure $\sigma=\sigma(\mu)$, $\supp{\sigma}=\Delta$, such that (cf.~\eqref{4})
\begin{equation}
\frac1n\chi(P_n(\cdot;\sigma))\overset{*}\to \mu,
\quad n\to\infty.
\label{5}
\end{equation}
\end{proposition}

\subsection{}\label{s1s2}

For the proof of Proposition~\ref{pro1} we use the results which concerns the existence and properties of the weighted Leja points (see~\cite{Lej57} and~\cite[Chapter V]{SaTo97}).

Recall that Leja points for an arbitrary compact set $K\subset\CC$ belong to the boundary  $\partial K$ of $K$. They produce a sequence  $\{z_n\}_{n=1}^\infty$ and are defined by induction in the following way. Once the points  $z_1,\dots,z_n\in\partial K$ are already defined, the point $z_{n+1}\in \partial K$ is defined (not uniquely) from the relation
\begin{equation}
\prod_{j=1}^n|z_{n+1}-z_j|=\max_{z\in K}\prod_{j=1}^n|z-z_j|;
\label{6}
\end{equation}
when there are a few points satisfying the relation~\eqref{6}  we choose an arbitrary point from this set as $z_{n+1}$.

The following result was proven by F.~Leja~\cite{Lej57}.

\begin{theoremL}[F.~Leja~\cite{Lej57}]
Let $K\subset\CC$ be a regular compact set such that $\myh\CC\setminus{K}$  is a domain. Let $\{z_n\}$ be Leja points for $K$. Set $P_n(z):=\prod\limits_{k=1}^n(z-z_k)$. Then the following relation holds
\begin{equation}
\lim_{n\to\infty}\log|P_n(z)|^{1/n}=	g_K(z,\infty)-\gamma_K,
\quad z\notin{K},
\label{7}
\end{equation}
where $g_K(z,\infty)$ is the Green function for  the domain $\myh\CC\setminus{K}$, $\gamma_K$ is the Robin constant for~$K$.
\end{theoremL}

Following the monograph~\cite[Chapter V]{SaTo97} we define the $w$-weighted Leja points $\{x_n\}$ for the compact set $\Delta=[-1,1]$, where $w(x)=e^{V^\mu(x)}$ .Once the points  $x_1,\dots,x_n\in\Delta$ are already defined, the point $x_{n+1}\in\Delta$ is defined (non uniquely) from the relation:
\begin{equation}
e^{nV^\mu(x_{n+1})}\prod_{j=1}^n|x_{n+1}-x_j|=\max_{x\in \Delta}\(e^{nV^\mu(x)}\prod_{j=1}^n|x-x_j|\);
\label{8}
\end{equation}
as before when there are a few points satisfying to the relation~\eqref{8}  we choose an arbitrary point from this set as $x_{n+1}$.

In~\cite [Chapter V, p.~258, Theorem 1.1]{SaTo97} a general theorem on the limit properties of weighted Leja points was proven. From this theorem it follows  that for the sequence $\{x_n\}$ which satisfies the condition~\eqref{8} the following relation holds
\begin{equation}
\frac1n\sum_{k=1}^n\delta_{x_k}\overset{*}\to \mu,\quad n\to\infty.
\label{9}
\end{equation}
Thus we have that
\begin{equation}
\lim_{n\to\infty}\log\biggl|\prod_{k=1}^n(z-x_k)\biggr|^{1/n}=-V^\mu(z),
\quad z\in\myh\CC\setminus\Delta.
\label{9.2}
\end{equation}

%\subsection{}\label{s1s3}

\begin{proof}[of Proposition~\ref{pro1}]
Let $q\in(0,1/2)$. Below we construct a sequence of positive numbers $\eps=\{\eps_n\}_{n=1}^\infty$  with the following properties:

1) $\eps_n\downarrow0$ as $n\to\infty$,

2) $\sum\limits_{n=1}^\infty\eps_n<\infty$,

3) $\sum\limits_{k=n+1}^\infty\eps_k<\eps_n$.

Notice that the above conditions 1)--3) are fulfilled in particular for the sequence $\eps_n=q^{n^2}$, $q\in(0,1/2)$. This fact will also be used in the future.

Let fix some $n_0\in\NN$ and let $\eps_1,\dots,\eps_{n_0}$  be some positive numbers. Let $\{x_n\}$ be a sequence of $w$-weighted Leja points for $\Delta$ in the sense of the definition~\eqref{8} and~\eqref{9}.
Set $\sigma_n:=\sum\limits_{k=1}^n\eps_k\delta_{x_k}$ for $n=n_0$. It is clear that $P_n(x;\sigma_n)=(x-x_1)\dots(x-x_n)$.
Let now $\delta_n:=\min\limits_{1\leq k<s\leq n}|x_k-x_s|$.
Let choose a number $\eps_{n+1}\in(0,q^{n^2}\eps_n)$ such that for each measure  $\beta$ of the type $\beta=\sigma_n+2\eps_{n+1}\nu=\sum\limits_{k=1}^n\eps_k\delta_{x_k}+2\eps_{n+1}\nu$, $\supp{\nu}\subset\Delta$, $\nu(1)\leq1$ the following property holds. If $P_n(x;\beta)=\prod\limits_{k=1}^n(x-\myt{x}_k)$ is the $n$\,th polynomials orthogonal with respect to $\beta$, then
\begin{equation}
|x_k-\myt{x}_k|<\frac12\min\{q^{n^2},\delta_n\},\quad k=1,\dots,n,
\label{10}
\end{equation}
uniformly for all measures $\nu$, $\supp{\nu}\subset\Delta$, $\nu(1)\leq1$.
This can always be done because the $n$\,th orthogonal polynomial $P_n(x;\beta)$ depends on the first $2n$ moments of the measure $\beta$ only.

Now set $\sigma_{n+1}:=\sum\limits_{k=1}^{n+1}\eps_k\delta_{x_k}$
and let continue the above procedure further. As the result we will get a measure
\begin{equation}
\sigma=\sum_{n=1}^\infty\eps_n\delta_{x_n},
\label{11}
\end{equation}
with the following properties. Measure $\sigma$ supported on $\Delta$, $\supp{\sigma}=\Delta$, and for each $n\geq n_0$ for the zeros $x_{n,k}$, $k=1,\dots,n$ of the corresponding $n$\,th orthogonal polynomial $P_n(x;\sigma)=x^n+\dotsb=(x-x_{n,1})\dotsb(x-x_{n,n})$ the following relation holds
\begin{equation}
|x_k-x_{n,k}|<q^{n^2},\quad k=1,\dots,n.
\label{12}
\end{equation}

Hence, we have uniformly on $z\in K\Subset\CC\setminus\Delta$
\begin{align}
\log|P_n(z;\sigma)|
&=\sum_{k=1}^n\log|z-x_k+(x_k-x_{n,k})|\notag\\
&=\sum_{k=1}^n\(\log|z-x_k|+\log|1+(x_k-x_{n,k})/(z-x_k)|\)
\notag\\
&=\log\biggl|\prod_{k=1}^n(z-x_k)\biggr|+\sum_{k=1}^n\log|1+O(q^{n^2})|.
\label{14}
\end{align}
From this it follows that uniformly on $z\in K\Subset\CC\setminus\Delta$  (cf.~\eqref{7}) 
\begin{equation}
\lim_{n\to\infty}
\log|P_n(z;\sigma)|^{1/n}=-V^{\mu}(z).
\label{15}
\end{equation}
Therefore as $n\to\infty$
$$
\frac1n\chi\bigl(P_n(\cdot;\sigma)\bigr)\overset{*} \to\mu.
$$
Relation~\eqref{5} is proven.
\end{proof}

\section{}\label{s2}

In this section we construct two  examples of compact sets and sequences of discrete measures supported on those compact sets with the following property. These sequences converge in weak-$*$ topology to the equilibrium measures for the compact sets but the corresponding sequences of the logarithmic potentials do not converge to the equilibrium potentials in any neighbourhood of the compact sets.

The construction is based on the following well-known property of the logarithmic capacity (see.~\cite[Chapter~VII, \S~1, Theorem~2]{Gol66}).
Let $P_n(z)=z^n+\dotsb$ be an arbitrary monic polynomial of degree~${n}$. Then for each non generated compact set $F\subset\CC$ and the corresponding $E=\left\{z:P_n(z)\in{F}\right\}$ we have that
\begin{equation}
\mycap(E)=\mycap(F)^{1/n}.
\label{Gol1}
\end{equation}
In particularly for the case of a disk $F=\{w:|w|\leq{\rho^n}\}$ we obtain
$$
\mycap\left\{z:\bigl|P_n(z)\bigr|\leq{\rho^n}\right\}=\rho.
$$

The following statement was given by Herbert Stahl in~\cite[p.~166]{Sta97} but without details. 
\begin{proposition}\label{sta1}
Let $\Gamma=\{z:|z|=1\}$  be the unit circle, $\{\mu_n\}$ be the sequence of probability measures associated with the zeros of  monic polynomials $\{z^n-1\}$:
$$
\mu_n=\frac1n\sum_{j=1}^n\delta_{z_{n,j}}
$$
($\{z_{n,j}\}$ are the roots of degree $n$ from $1$).
Let $\lambda_\Gamma$ be the equilibrium measure for the unit circle~$\Gamma$. Then  $\mu_n\overset{*}\to\lambda_\Gamma$ as $n\to\infty$,~but
$$
V^{\mu_n}(z)\overset{\mycap}{\not\rightarrow} V^{\lambda_\Gamma}(z)
$$
in any neighbourhood of the compact set $\Gamma$.
\end{proposition}

\begin{proof}
Let fix an arbitrary $\eps>0$ and consider the set
$$
E_n:=\left\{z:|z|\geq1,
\bigl|V^{\mu_n}(z)-V^{\lambda_\Gamma}(z)\bigr|\geq\eps\right\},\quad\text{where}
\quad V^{\lambda_\Gamma}(z)=\log\frac1{|z|},\quad|z|\geq1.
$$
We have that
\begin{equation}
V^{\mu_n}(z)-V^{\lambda_\Gamma}(z)
=\frac1{n}\sum_{j=1}^n\log\frac1{|z-z_{n,j}|}-\log\frac1{|z|}\notag\\
=\frac1{n}\log\frac{|z|^n}{|z^n-1|}.
\notag
\end{equation}
Therefore
\begin{align}
E_n&=\left\{z:|z|\geq1,\biggl|\frac1{n}\log\frac{|z^n|}{|z^n-1|}\biggr|\geq\eps\right\}
\supset\left\{z:|z|\geq1,\frac1{n}\log\frac{|z^n|}{|z^n-1|}\geq\eps\right\}
\notag\\
&=\left\{z:|z|\geq1,\,|z^n-1|\leq|z^n|{e^{-n\eps}}\right\}
\supset\left\{z:|z^n|\geq1,\,|z^n-1|\leq{e^{-n\eps}}\right\}\notag\\
&=\left\{z:p_n(z)\in F_n\right\}=:\myt{E}_n,\notag
\end{align}
where $F_n=\{w:|w|\geq1,|w-1|\leq{e^{-n\eps}}\}$, $w=p_n(z)\equiv{z^n}$.
Let  $\rho>1$ and $K_\rho:=\{z: |z|\leq\rho\}$. Then for  all natural $n$ large enough, $n\geq n_0$,  we have that $\myt{E}_n\subset  K_\rho$. Hence $E_n\cap K_\rho\supset \myt{E}_n$ for $n\geq n_0$. Therefore  from~\eqref{Gol1} it follows that
$$
\mycap(E_n\cap K_\rho)\geq\mcap{\myt{E}_n}=\mycap(F_n)^{1/n}.
$$
It is easy to see that $\displaystyle\frac14e^{-n\eps}<\mycap(F_n)<e^{-n\eps}$. Thus
$$
\mycap\(E_n\cap K_\rho\)\geq\(\frac14\)^{1/n}e^{-\eps}\not\rightarrow0\quad \text{as $n\to0$}.
$$
\end{proof}

\begin{proposition}\label{sta2}
Let $K=[-1,1]$, $\{\mu_n\}$ be the sequence of probability measures associated with Chebyshev polynomials of first kind for $K$. Let $\lambda_K$ be the equilibrium measure for the segment~$K$.
Then $\mu_n\overset{*}\to\lambda_K$ as $n\to\infty$, but
$$
V^{\mu_n}(z)\overset{\mycap}{\not\rightarrow} V^{\lambda_K}(z)
$$
in any neighbourhood of the compact set $K$.
\end{proposition}

\begin{proof}
Let fix an arbitrary $\eps>0$ and consider the set
$$
E_n=\left\{z:\bigl|V^{\mu_n}(z)-V^{\lambda_K}(z)\bigr|\geq\eps\right\},
$$
where $\displaystyle V^{\lambda_K}(z)=\log\frac2{|\varphi(z)|}$, $\displaystyle\varphi(z)=z+(z^2-1)^{1/2}$. We obtain that
\begin{align}
E_n&=\left\{z:\biggl|\frac1{n}\log\frac1{|T_n(z)|}-\log\frac2{|\varphi(z)|}\biggr|\geq\eps\right\}
=\left\{z:\biggl|\frac1{n}\log\biggl|\frac{\varphi(z)^n}{2^nT_n(z)}\biggr|\biggr|\geq\eps\right\}\notag\\
&\supset\left\{z:\frac1{n}\log\biggl|\frac{\varphi(z)^n}{2^nT_n(z)}\biggr|\geq\eps\right\}
=\left\{z:\biggl|\frac{\varphi(z)^n}{2^nT_n(z)}\biggr|\geq e^{n\eps}\right\}
\notag\\
&=\left\{z:\biggl|\frac{2^nT_n(z)}{\varphi(z)^n}\biggr|\leq e^{-n\eps}\right\}
=\left\{z:\bigl|T_n(z)\bigr|\leq 2^{-n}|\varphi(z)^n|e^{-n\eps}\right\}
\notag\\
&\supset\left\{z:\bigl|T_n(z)\bigr|\leq 2^{-n}e^{-n\eps}\right\}=
\left\{z:T_n(z)\in F_n\right\}=:\myt{E}_n,
\notag
\end{align}
where $F_n:=\{w:|w|\leq 2^{-n}e^{-n\eps}\}$.

Let $K_\rho:=\{z:|\varphi(z)|\leq\rho\}$, $\rho>1$, be the closer of an ellipse.  It is well-known that for the Chebyshev polynomials $T_n(z)=z^n+\dotsb$  the explicit representation is valid
\begin{equation}
T_n(z)=\frac1{2^n}\(\varphi(z)^n+\varphi(z)^{-n}\),\quad z\notin K.
\notag
\end{equation}
From this it follows that $\myt{E}_n\subset K_\rho$ for all $n$ large enough, $n\geq n_0$. Hence $E_n\cap K_\rho\supset\myt{E}_n$ for $n\geq n_0$.
Therefore from~\eqref{Gol1} it follows that
$$
\mycap(E_n\cap K_\rho)\geq\mcap(\myt{E}_n)=\mycap(F_n)^{1/n}.
$$
Since $\mycap(F_n)^{1/n}=e^{-\eps}/2$ then
$$
\mycap\(E_n\cap K_\rho\)\geq\frac12e^{-\eps}\not\rightarrow0\quad \text{as
$n\to0$}.
$$
\end{proof}

\end{document}